\documentclass[11pt,oneside,reqno]{amsart}

\usepackage[a4paper, total={450pt,675pt}]{geometry}
\usepackage[OT2,T1]{fontenc}
\usepackage[utf8]{inputenc}
\usepackage{amssymb}

\usepackage[dvipsnames]{xcolor}
\usepackage[
    colorlinks=true,
    linkcolor=Maroon,
    citecolor=JungleGreen,
    urlcolor=NavyBlue]{hyperref}

\usepackage[capitalize,nameinlink]{cleveref}

\usepackage{lipsum}

\usepackage{enumitem}
\usepackage{dsfont}
\usepackage{tikz}
\usepackage{subcaption}
\usetikzlibrary{arrows}

\makeatletter
\def\@tocline#1#2#3#4#5#6#7{\relax
  \ifnum #1>\c@tocdepth 
  \else
    \par \addpenalty\@secpenalty\addvspace{#2}%
    \begingroup \hyphenpenalty\@M
    \@ifempty{#4}{%
      \@tempdima\csname r@tocindent\number#1\endcsname\relax
    }{%
      \@tempdima#4\relax
    }%
    \parindent\z@ \leftskip#3\relax
    \advance\leftskip\@tempdima\relax
    \rightskip\@pnumwidth plus4em \parfillskip-\@pnumwidth
    #5\leavevmode\hskip-\@tempdima
      \ifcase #1
       \or\or \hskip 2em \or \hskip 2em \else \hskip 3em \fi%
      #6\nobreak\relax
    \dotfill\hbox to\@pnumwidth{\@tocpagenum{#7}}\par
    \nobreak
    \endgroup
  \fi}
\makeatother

\usepackage[
    backend=biber,
    style=alphabetic,
    maxbibnames=10,
    sorting=nyt,
    giveninits=true]{biblatex}

\DeclareLabelalphaTemplate{
  \labelelement{
    \field[final]{shorthand}
    \field{label}
    \field[strwidth=3]{labelname}}
  \labelelement{
    \field[strwidth=2,strside=right]{year}}
}

\DeclareFieldFormat[article,book,incollection,misc]{title}{#1}
\DeclareFieldFormat[inbook,incollection]{booktitle}{#1}
\DeclareFieldFormat[misc]{date}{\textit{preprint} (#1)}
\DeclareFieldFormat{pages}{#1}
\DeclareFieldFormat{doi}{
\href{http://www.ams.org/mathscinet-getitem?mr=#1}{#1}}

\renewbibmacro{in:}{%
  \ifentrytype{article}{}{\printtext{\bibstring{in}\intitlepunct}}}

\newtheorem{theorem}{Theorem}[section]

\newtheorem{lemma}[theorem]{Lemma}
\newtheorem{proposition}[theorem]{Proposition}
\newtheorem{corollary}[theorem]{Corollary}

\newtheorem{remark}[theorem]{Remark}

\crefname{section}{Sect.}{section}
\numberwithin{equation}{section}

\DeclareMathOperator{\const}{const.}

\newcommand\Bprime{B^{\hspace{.2mm}\prime}}
\newcommand{\co}{\operatorname{co}}
\newcommand\dprime{d^{\hspace{.3mm}\prime}}
\newcommand\deltaprime{\delta^{\hspace{.2mm}\prime}}
\newcommand\deltasecond{\delta^{\hspace{.2mm}\prime\prime}}
\renewcommand\epsilon{\varepsilon}
\newcommand\GammaRprime{\Gamma_{\!R}^{\hspace{.2mm}\prime}}
\newcommand\Gprime{G^{\hspace{.2mm}\prime}}
\newcommand\N{\mathbb{N}}
\newcommand\NRprime{N_{\hspace{-.2mm}R}^{\hspace{.2mm}\prime}}
\renewcommand\phi{\varphi}
\newcommand\Pprime{P^{\hspace{.4mm}\prime}}
\newcommand\Psecond{P^{\hspace{.4mm}\prime\prime}}
\newcommand\dsecond{d^{\hspace{.4mm}\prime\prime}}

\newcommand\msb{\hspace{-.5mm}}
\newcommand\msf{\hspace{.5mm}}
\newcommand\ssb{\hspace{-.3mm}}
\newcommand\ssf{\hspace{.3mm}}
\newcommand\vsb{\hspace{-.1mm}}
\newcommand\vsf{\hspace{.1mm}}

\begin{document}
\title[Bottom of the $L^2$ spectrum of the Laplacian 
on locally symmetric spaces]{Bottom of the $L^2$ spectrum of 
the Laplacian \\ on locally symmetric spaces}

\author{Jean-Philippe ANKER and Hong-Wei ZHANG}


\makeatletter
\@namedef{subjclassname@2020}{\textnormal{2020}
    \it{Mathematics Subject Classification}}
\makeatother

\maketitle

\vspace{-10pt}
\begin{center}
\footnotesize\it
    In memory of Michel Marias (1953-2020), 
    who introduced us to locally symmetric spaces 
\end{center}

\begin{abstract}
We estimate the bottom of the $L^2$ spectrum of the Laplacian on locally symmetric spaces in terms of the critical exponents of appropriate Poincar\'e series.
Our main result is the higher rank analog of a characterization
due to Elstrodt, Patterson, Sullivan and Corlette in rank one.
It improves upon previous results obtained by Leuzinger and Weber in higher rank.
\end{abstract}


\section{Introduction}\label{intro}
We adopt the standard notation and refer to \cite{Hel2000} for more details. Let $G$ be a semi-simple Lie group, connected, noncompact, with finite center, and $K$ be a maximal compact subgroup of $G$. The homogeneous space $X = G/K$ is a Riemannian symmetric space of noncompact type. Let $\mathfrak{g} = \mathfrak{k} \oplus \mathfrak{p}$ be the Cartan decomposition of the Lie algebra of $G$. The Killing form of $\mathfrak{g}$ induces a $K$-invariant inner product $\langle\,.\,,\,.\,\rangle$ on $\mathfrak{p}$, hence a $G$-invariant Riemannian metric on $G/K$. Fix a maximal abelian subspace $\mathfrak{a}$ in $\mathfrak{p}$. We identify $\mathfrak{a}$ with its dual $\mathfrak{a}^{*}$ by means of the inner product inherited from $\mathfrak{p}$.
Let $\Gamma$ be a discrete torsion-free subgroup of $G$ that acts freely and properly discontinuously on $X$. Then $Y\!=\Gamma\backslash X$ is a locally symmetric space, whose Riemannian structure is inherited from $X$. We denote by $d(\,.\,,\,.\,)$ the joint Riemannian distance on $X$ and $Y$,  by $n$ their joint dimension, and by $\ell$ their joint rank, which is the dimension of $\mathfrak{a}$.

Let $\Sigma\subset\mathfrak{a}$ be the root system of $(\mathfrak{g},\mathfrak{a})$ and let $W$ be the associated Weyl group. Choose a positive Weyl chamber $\mathfrak{a}^+\subset\mathfrak{a}$ and let $\Sigma^+{\subset\Sigma}$ be the corresponding subsystem of positive roots. Denote by $\rho= \frac{1}2 \sum_{\alpha\in\Sigma^+}m_{\alpha} \alpha$ the half sum of positive roots counted with their multiplicities. Occasionally we shall need the reduced root system
$\Sigma_{\text{red}}=\lbrace\alpha\in\Sigma\,|\,\frac\alpha2\notin\Sigma\rbrace$.\\

Consider the classical Poincar\'e series
\begin{align}\label{PSeries}
P_{s}(xK,yK) =\sum\nolimits_{\gamma\in\Gamma}e^{-sd(xK,\gamma yK)}
\qquad\forall\,s>0,\,\forall\,x,y\in G
\end{align}
and denote by
$\delta(\Gamma)=\inf\lbrace s>0\,|\,P_s(xK,yK) <+\infty\rbrace$
its critical exponent, which is independent of $xK$ and $yK$.
Recall that $\delta(\Gamma)\in[0,2\hspace{.2mm}\|\rho\|]$ may be also defined by
\begin{align*}
\delta(\Gamma)=\limsup_{R\to+\infty}\frac{\log N_{R}(xK,yK)}R
\qquad\forall\,x,y\in G,
\end{align*}
where $N_{R}(xK,yK)=|\lbrace\gamma\in\Gamma\,|\,d(xK,\gamma yK)\le R\rbrace|$ denotes the orbital counting function. Finally, let $\Delta_Y$ be the Laplace-Beltrami operator on $Y$ and let $\lambda_0(Y)$ be the bottom of the $L^2$ spectrum of $-\Delta_Y$.
The following celebrated result, due to Elstrodt (\cite{Els1973a}, \cite{Els1973b}, \cite{Els1974}), Patterson \cite{Pat1976}, Sullivan \cite{Sul1987} and Corlette \cite{Cor1990}, expresses $\lambda_0(Y)$ in terms of $\rho$ and $\delta(\Gamma)$ in rank one.

\begin{theorem}\label{corlette}
In the rank one case ($\ell=1$), we have
\begin{align}\label{rank 1}
\lambda_0(Y) =
\begin{cases}
\|\rho\|^2
&\qquad\textnormal{if}\;\;0\le\delta(\Gamma)\le\|\rho\|,\\
\|\rho\|^2-(\delta(\Gamma)-\|\rho\|)^2
&\qquad\textnormal{if}\;\;\|\rho\|\le\delta(\Gamma)\le2\hspace{.2mm}\|\rho\|.
\end{cases}
\end{align}
\end{theorem}

This result was extended in higher rank as follows by Leuzinger \cite{Leu2004} and Weber \cite{Web2008}.
Let $\rho_{\text{min}}=\min_{H\in\overline{\mathfrak{a}^+},\|H\|=1}\langle\rho,H\rangle\in(0,\|\rho\|]$.
Notice that $\rho_{\text{min}}={\|\rho\|}$ in rank one and thus the following theorem reduces to Theorem \ref{corlette}.

\begin{theorem}\label{leuzinger}
In the general case ($\ell\ge1$), the following estimates hold:
\begin{itemize}[leftmargin=*]
\item[$\bullet$]
Upper bound:
\begin{align*}
\lambda_0(Y)\le
\begin{cases}
\,\|\rho\|^2
&\qquad\textnormal{if}\;\;0\le\delta(\Gamma)\le\|\rho\|,\\
\,\|\rho\|^2-(\delta(\Gamma)-\|\rho\|)^2
&\qquad\textnormal{if}\;\;\|\rho\|\le\delta(\Gamma)\le2\hspace{.2mm}\|\rho\|.
\end{cases}
\end{align*}
\item[$\bullet$]
Lower bound:
\begin{align*}
\lambda_0(Y)\ge
\begin{cases}
\,\|\rho\|^2
&\qquad\textnormal{if}\;\;0\le\delta(\Gamma)\le\rho_{\text{\rm min}},\\
\,\max\hspace{.5mm}\lbrace0,\|\rho\|^2-(\delta(\Gamma)-\rho_{\text{\rm min}})^2\rbrace
&\qquad\textnormal{if}\;\;\rho_{\text{\rm min}}\le\delta(\Gamma)\le2\hspace{.2mm}\|\rho\|.
\end{cases}
\end{align*}
\end{itemize}

\noindent In other terms,
\begin{align*}
\begin{cases}
\,\lambda_0(Y)=\|\rho\|^2
&\textnormal{if}\;\;\delta(\Gamma)\in\bigl[\hspace{.2mm}0,\rho_{\text{\rm min}}\bigr],\\[5pt]
\,\lambda_0(Y)\in
\bigl[\hspace{.2mm}\|\rho\|^2\!-(\delta(\Gamma)\hspace{-.5mm}-\hspace{-.5mm}\rho_{\text{\rm min}})^2,
\|\rho\|^2\hspace{.2mm}\bigr]
&\textnormal{if}\;\;\delta(\Gamma)\in\bigl[\hspace{.3mm}\rho_{\text{\rm min}},\|\rho\|\bigr],\\[5pt]
\,\lambda_0(Y)\in
\bigl[\hspace{.2mm}\|\rho\|^2\!-\hspace{-.5mm}(\delta(\Gamma)\hspace{-.5mm}-\hspace{-.5mm}\rho_{\text{\rm min}})^2,
\|\rho\|^2\!-\hspace{-.5mm}(\delta(\Gamma)\hspace{-.5mm}-\hspace{-.5mm}\|\rho\|)^2\hspace{.2mm}\bigr]
&\textnormal{if}\;\;\delta(\Gamma)\in\bigl[\hspace{.2mm}\|\rho\|,\|\rho\|\hspace{-.5mm}+\hspace{-.5mm}\rho_{\text{\rm min}}\bigr],\\[5pt]
\,\lambda_0(Y)\in\bigl[\hspace{.2mm}0,
\|\rho\|^2\!-\hspace{-.5mm}(\delta(\Gamma)\hspace{-.5mm}-\hspace{-.5mm}\|\rho\|)^2\hspace{.2mm}\bigr]
&\textnormal{if}\;\;\delta(\Gamma)\in\bigl[\hspace{.2mm}
\|\rho\|\hspace{-.5mm}+\hspace{-.5mm}\rho_{\text{\rm min}},2\hspace{.2mm}\|\rho\|\hspace{.2mm}\bigr].
\end{cases}  
\end{align*}
\end{theorem}

\vspace{10pt}
In this paper, we first improve the lower bound of $\lambda_0(Y)$ in Theorem \ref{leuzinger} by a slight modification of the classical Poincar\'e series \eqref{PSeries}. Let $\deltaprime{(\Gamma)}$ denote
the critical exponent of the modified Poincaré series \eqref{PprimeSeries}
associated to the polyhedral distance \eqref{d'}.

\begin{theorem}\label{TheoremFirstImprovement}
The following lower bound holds for the bottom $\lambda_0(Y)$
of the $L^2$ spectrum of $-\Delta$ on $Y\!=\Gamma\backslash G/K$:
\begin{align*}
\lambda_0(Y)\ge
\begin{cases}
\,\|\rho\|^2
&\qquad\textnormal{if}\;\;0\le\deltaprime{(\Gamma)}\le\|\rho\|,\\
\,\|\rho\|^2-(\deltaprime{(\Gamma)}-\|\rho\|)^2
&\qquad\textnormal{if}\;\;\|\rho\|\le\deltaprime{(\Gamma)}\le2\hspace{.2mm}\|\rho\|.
\end{cases}
\end{align*}
\end{theorem}

We obtain next a plain analog of Theorem \ref{corlette} by considering a more involved family of Poincar\'e series. We denote by $\deltasecond{(\Gamma)}$ the critical exponent of 
$\Psecond_s(xK, yK)$, see \eqref{PsecondSeries}.

\begin{theorem}\label{second improvement}
The following characterization holds for the bottom $\lambda_0(Y)$
of the $L^2$ spectrum of $-\Delta$ on $Y\!=\Gamma\backslash G/K$:
\begin{align}\label{higher rank charac}
\lambda_0(Y)=\begin{cases}
\,\|\rho\|^2
&\qquad\textnormal{if}\;\;0\le\deltasecond{(\Gamma)}\le\|\rho\|,\\
\,\|\rho\|^2- (\deltasecond{(\Gamma)}-\|\rho\|)^2
&\qquad\textnormal{if}\;\;\|\rho\|\le\deltasecond{(\Gamma)}\le 2\hspace{.2mm}\|\rho\|.
\end{cases}
\end{align}
\end{theorem}

\begin{remark}
If $\Gamma$ is a lattice, i.e., $Y\!=\Gamma\backslash G/K$ has finite volume,
then $\lambda_0(Y)=0$ and $\deltasecond(\Gamma)=2\hspace{.2mm}\|\rho\|$, hence $\deltaprime(\Gamma)=2\hspace{.2mm}\|\rho\|$.
Furthermore $\delta(\Gamma)=2\hspace{.2mm}\|\rho\|$ \cite[Theorem~7.4]{Alb1999}.
As pointed out by Corlette \cite{Cor1990} in rank one and by Leuzinger \cite{Leu2003} in higher rank,
if $G$ has Kazhdan's property (T), then the following conditions are actually equivalent:
\vspace{1mm}

\centerline{
(a) $\Gamma$ is a lattice,\quad
(b) $\lambda_0(Y)=0$,\quad
(c) $\delta(\Gamma)=2\hspace{.2mm}\|\rho\|$,\quad
(d) $\deltasecond(\Gamma)=2\hspace{.2mm}\|\rho\|$.
}
\end{remark}

\begin{remark}
As for the Green function, the heat kernel
\begin{align*}
h_t^Y(\Gamma xK,\Gamma yK)=\sum\nolimits_{\gamma\in\Gamma}h_t{(Ky^{-1}\gamma^{-1}xK)}
\end{align*}
on a locally symmetric space $Y\!=\Gamma\backslash G/K$ can be expressed and estimated by using the heat kernel $h_t$ on the symmetric space $X =G/K$, whose behavior is well understood \cite{AnJi1999,AnOs2003}.
By adapting straightforwardly the methods carried out in \cite{DaMa1988,Web2008}, and by applying Theorem \ref{second improvement} instead of Theorem \ref{corlette} and Theorem \ref{leuzinger}, we refine the Gaussian bounds of $h_{t}^{Y}$ and get rid in particular of 
$\rho_{\text{\rm min}}$.
The following estimates hold for all {$t>0$} and all $x,y\in G$:\\
\noindent{\rm (i)}
Assume that $\deltasecond{(\Gamma)}<\|\rho\|$
and let $\deltasecond{(\Gamma)}<s<\|\rho\|$.
Then 
\begin{align*}
h_t^Y( \Gamma xK,\Gamma yK)
\lesssim t^{-\frac n2}\,(1+t)^{\frac{n-D}2}\,e^{-\|\rho\|^2t}\,
e^{-\frac{d (\Gamma xK,\Gamma yK)^2}{4t}}\,\Psecond_s(xK, yK).
\end{align*}

\noindent{\rm (ii)}
Assume that $\|\rho\|\le\deltasecond{(\Gamma)}<2\hspace{.2mm}\|\rho\|$
and let $\deltasecond{(\Gamma)}-\|\rho\|<s_1<s_2<\|\rho\|$.
Then
\begin{align*}
h_t^Y(\Gamma xK,\Gamma yK)
\lesssim t^{-\frac n2}\,e^{-\left(\|\rho\|^2-s_2^2\right)t}\,\Psecond_{\|\rho\|+s_1}(xK,yK).
\end{align*}

\noindent{\rm (iii)}
Assume that $\deltasecond{(\Gamma)}<2\hspace{.2mm}\|\rho\|$.
Let $s>\deltasecond{(\Gamma)}$ and $\epsilon>0$. Then
\begin{align*}
h_t^Y( \Gamma xK,\Gamma yK)
\lesssim t^{-\frac n2}\,
e^{-\left({\|\rho\|^2-(\deltasecond(\Gamma)-\|\rho\|)^2}-2\epsilon\right)t}\,
e^{-\frac{d(\Gamma xK,\Gamma yK)^2}{4(1+\epsilon)t}}\,\Psecond_s(xK, xK)^{\frac{1}{2}}\Psecond_s(yK, yK)^{\frac{1}{2}}.
\end{align*}
\end{remark}
\vspace{10pt}
\section{First improvement}

In this section, we replace the Riemannian distance $d$ on $X$ by a \textit{polyhedral} distance
\begin{align}\label{d'}
\textstyle
\dprime(xK,yK)=\langle\frac{\rho}{\|\rho\|},(y^{-1}x)^+\rangle
\qquad\forall\,x,y\in G,
\end{align}
where $(y^{-1}x)^{+}$ denotes the $\overline{\mathfrak{a}^+}$-component of $y^{-1}x$ in the Cartan decomposition $G=K(\exp\overline{\mathfrak{a}^+})K$.
The corresponding balls, which reflect the volume growth of $X$ at infinity, played an important role in \cite{Ank1990} and \cite{Ank1991}. More general polyhedral sets were considered in \cite{Ank1992} and \cite{AAS2010}.

\begin{proposition}
$\dprime$ is a $G$-invariant distance on $X$.
\end{proposition}

\begin{proof}
Notice first that \eqref{d'} descends from $G\times{G}$ to $X\times{X}$, as the map $z\mapsto{z^{+}}$ is $K$-bi-invariant on ${G}$.
The $G$-invariance of $\dprime$ is straightforward from the definition \eqref{d'}. The symmetry
\begin{align*}
\dprime(xK,yK)=\dprime(yK,xK)
\end{align*}
follows from
\begin{align}\label{distance properties}
(y^{-1})^+=-w_0. y^+
\quad\text{and}\quad
-w_{0}.\rho=\rho,
\end{align}
where $w_0$ denotes the longest element in the Weyl group. Let us check the triangular inequality
\begin{align*}
\dprime(xK,yK)\le\dprime(xK,zK)+\dprime(zK,yK). 
\end{align*}
By $G$-invariance, we may reduce to the case where $zK = eK$.
According to Lemma \ref{distance lemma} below,
\begin{align*}
x^{+} + (y^{-1})^{+} - (y^{-1}x)^{+}
\end{align*}
belongs to the cone generated by the positive roots.
\end{proof}

\begin{lemma}\label{distance lemma}
For every $x,y \in G$, we have the following inclusion
\begin{align}\label{co}
\co[W.(xy)^+]\subset\co[W.(x^++ y^+)]
\end{align}
between convex hulls.
\end{lemma}

\begin{proof}
The inclusion \eqref{co} amounts to the fact that
\begin{align*}
x^{+} + y^{+} - (xy)^{+}
\end{align*}
belongs to the cone generated by the positive roots or, equivalently, to the inequality
\begin{align}\label{co2}
\langle\lambda,(xy)^+\rangle\le\langle\lambda,x^+\rangle+\langle\lambda,y^+\rangle
\qquad\forall\,\lambda\in\overline{\mathfrak{a}^{+}}.
\end{align}
It is enough to prove \eqref{co2} for all highest weights $\lambda$ of irreducible finite-dimensional complex representations $\pi:G\longrightarrow\text{GL}(V)$ with $K$-fixed vectors. 
According to Weyl’s unitary trick (see for instance \cite[Proposition 7.15]{Kna2002}), there exists an inner product on $V$ such that
\begin{align*}
\textstyle
\begin{cases}
\pi(k)\text{ is unitary}&{\forall}\,k\in K,\\
\pi(a)\text{ is self-adjoint}&{\forall}\,a\in\exp\mathfrak{a}.
\end{cases}
\end{align*}
As $\lambda$ is the highest weight of $\pi$, then
\begin{align*}
e^{\langle \lambda, (xy)^{+} \rangle} = \| \pi(xy) \|
\le \| \pi(x) \| \| \pi(y) \| = 
e^{\langle \lambda, x^{+} \rangle} e^{\langle \lambda, y^{+} \rangle}
= e^{\langle \lambda, x^{+} + y^{+} \rangle}.
\end{align*}
\end{proof}

\begin{remark}
The distance $\dprime$ is comparable to the Riemannian distance $d$. Specifically,
\begin{align}\label{d and d'}
\textstyle
\frac{\rho_{\text{min}}}{\|\rho\|}d(xK,yK)\le\dprime(xK,yK)\le d(xK,yK)
\qquad\forall\,x,y\in G.
\end{align}
This follows indeed from
\begin{align*}
\textstyle
\frac{\rho_{\text{min}}}{\|\rho\|}\|H\|\le\langle\frac{\rho}{\|\rho\|},H\rangle\le\|H\|
\qquad\forall\,H\in\overline{\mathfrak{a}^+}.
\end{align*}
\end{remark}

\vspace{10pt}
The volume of balls
\begin{align*}
\Bprime_r(xK)=\lbrace yK\in X\,|\,\dprime(yK,xK)\le r\rbrace
\end{align*}
was determined in \cite[Lemma 6]{Ank1990}.
For the reader's convenience, we recall the statement and its proof.

\begin{lemma}\label{ball vol}
For every $x\in G$ and $r>0$, we have
\footnote{The symbol $f\asymp g$ between two non-negative expressions
means that there exist constants $0<A\le B<+\infty$ such that $Ag\le f\le Bg$.}
\begin{align*}
|\Bprime_r(xK)|\asymp
\begin{cases}
\,r^{n}
&\qquad\textnormal{if}\;\;0<r<1,\\
\,r^{\ell-1}e^{2\hspace{.2mm}\|\rho\|r}
&\qquad\textnormal{if}\;\;r\ge1.
\end{cases}
\end{align*}
\end{lemma}

\begin{remark}
Notice the different large scale behavior,
in comparison with the classical ball volume
\begin{align*}
|B_r(xK)|\asymp
\begin{cases}
\,r^{n}
&\qquad\textnormal{if}\;\;0<r<1,\\
\,r^{\frac{\ell-1}2}e^{2\hspace{.2mm}\|\rho\|r}
&\qquad\textnormal{if}\;\;r\ge1,
\end{cases}
\end{align*}
see for instance \cite{Str1981} or \cite{Kni1997}.
\end{remark}

\begin{proof}
By translation invariance, we may assume that $x=e$.
Recall the integration formula
\begin{align}\label{IntegrationFormula}
\int_{X}dx\,f(x)=\const\int_{K} dk \ 
\int_{\mathfrak{a}^+}dH\,\omega(H)\,f(k(\exp H)K),
\end{align}
in the Cartan decomposition, with density
\begin{align*}
\omega(H)
=\prod_{\alpha\in\Sigma^+}(\sinh\langle\alpha,H\rangle )^{m_\alpha}
\asymp\prod_{\alpha\in\Sigma^+}\!
\left(\tfrac{\langle\alpha,H\rangle}{1+\langle\alpha,H\rangle}\right)^{m_\alpha}
e^{2\langle\rho,H\rangle},
\end{align*}
since $\sinh{t}\sim\tfrac{t}{1+t}e^{t}$ for all $t\ge0$.
Thus
\begin{align*}
|\Bprime_r(eK)|
\,=\,\const\int_{\lbrace H\in\mathfrak{a}^+|\langle\rho,H\rangle\le\|\rho\|r\rbrace}dH\,\omega(H)
\,\asymp\,\int_{\lbrace H\in\mathfrak{a}^+|\|H\|\le r\rbrace}dH\,
\prod_{\alpha\in\Sigma^+}\langle\alpha,H\rangle^{m_\alpha}
\asymp r^n
\end{align*}
if $r$ is small.
Let us turn to $r$ large. On the one hand, we estimate from above
\begin{align*}
| \Bprime_r(eK) |\lesssim
\int_{\lbrace H\in\mathfrak{a}^+|\langle\rho,H\rangle\le\|\rho\|r\rbrace}dH\,e^{2\langle\rho,H\rangle}
\asymp \int_0^{2\hspace{.2mm}\|\rho\|r}ds\,s^{\ell-1 }e^s
\asymp r^{\ell -1} e^{2 \| \rho \| r}.
\end{align*}
On the other hand, let $H_0\in\mathfrak{a}^+$. As
\begin{align*}
\omega(H)\asymp e^{2\langle\rho,H\rangle}
\qquad\forall\,H\in H_0+\overline{\mathfrak{a}^+},
\end{align*}
we estimate from below
\begin{align*}
|\Bprime_r(eK) |\gtrsim
\int_{\lbrace H\in H_0+\mathfrak{a}^+|\langle\rho,H\rangle\le\|\rho\|r\rbrace}dH\,e^{2 \langle\rho,H\rangle}
\gtrsim\int_{C_0}^{2\hspace{.2mm}\|\rho\|r}ds\,s^{\ell-1 }e^s
\asymp r^{\ell-1} e^{2\hspace{.2mm}\|\rho\|r},
\end{align*}
where $C_{0} > 0$ is a constant depending on $H_0$.
\end{proof}

Consider now the modified Poincar\'e series
\begin{align}\label{PprimeSeries}
\Pprime_s(xK,yK)=\sum\nolimits_{\gamma\in\Gamma}e^{-s\dprime(xK,\gamma yK)}
\qquad\forall\,s>0,\,\forall\,x,y \in G
\end{align}
associated with $\dprime$, its critical exponent
\begin{align}\label{delta' def}
\deltaprime(\Gamma)=\inf\lbrace s>0\,|\,\Pprime_s(xK,yK) <+\infty\rbrace
\end{align}
and the modified orbital counting function
\begin{align}\label{N'CountingFunction}
\NRprime(xK,yK)=|\lbrace\gamma\in\Gamma\,|\,\dprime(xK,\gamma yK)\le R\rbrace|
\qquad \forall\,R\ge0,\,\forall\,x,y\in G.
\end{align}
The following proposition shows that \eqref{PprimeSeries}, \eqref{delta' def} and \eqref{N'CountingFunction} share the properties of their classical counterparts.

\begin{proposition}\label{delta' prop}
The following assertions hold:
\begin{enumerate}[leftmargin=8mm]
\item[\rm (i)]
$\deltaprime{(\Gamma)}$ does not depend on the choice of $x$ and $y$. 
\item[\rm (ii)]
$0\le\deltaprime{(\Gamma)}\le2\hspace{.2mm}\|\rho\|$. 
\item[\rm (iii)]
For every $x,y \in G$,
\begin{align}\label{limsup}
\deltaprime(\Gamma)=\limsup_{R\to+\infty}\frac{\log\NRprime(xK,yK)}R.
\end{align}
\end{enumerate}

\begin{remark}
It follows from \eqref{d and d'} that
\begin{align*}
P_s(xK,yK)\le\Pprime_s(xK,yK)\le P_{\frac{\rho_{\text{min}}}{\|\rho\|}s}(xK,yK)
\end{align*}
and
\begin{align*}
N_R(xK,yK)\le\NRprime(xK,yK)\le N_{\frac{\|\rho\|}{\rho_{\text{min}}}R}(xK,yK).
\end{align*}
Hence
\begin{align*}
0\le\delta{(\Gamma)}\le\deltaprime{(\Gamma)}
\le \frac{\|\rho\|}{\rho_{\text{min}}}\,\delta{(\Gamma)}.
\end{align*}
\end{remark}
\end{proposition}

\begin{proof}

\noindent(i) follows from the triangular inequality.
More precisely, let $x_1,y_1,x_2,y_2\in G$ and $s>0$.
Then
\begin{align*}
\dprime(x_2K,\gamma y_2K)\le\dprime(x_2K,x_1K)+\dprime(x_1K,\gamma y_1K)+
\underbrace{
\dprime(\gamma y_1K,\gamma y_2K)
}_{\dprime(y_1K,\,y_2K)}
\quad\forall\,\gamma\in\Gamma,
\end{align*}
hence
\begin{align*}
\underbrace{\sum\nolimits_{\gamma\in\Gamma}e^{-s\,\dprime(x_1K,\gamma y_1K)}}_{\Pprime_s(x_1K,y_1K)}
\le e^{s\lbrace\dprime(x_1K,x_2K)+\dprime(y_1K,y_2K)\rbrace}
\underbrace{\sum\nolimits_{\gamma\in\Gamma}e^{-s\,\dprime(x_2K,\gamma y_2K)}}_{\Pprime_s(x_2K,y_2K)}.
\end{align*}

\noindent (ii) According to (i), let us show, without loss of generality, that $\Pprime_s(eK,eK)<+\infty$ for every $s>2\hspace{.2mm}\|\rho\|$.
According to Lemma \ref{ball disjoint} below, there exists $r>0$ such that the balls
$\Bprime_{r}(\gamma K)$, with $\gamma\in\Gamma$, are pairwise disjoint in $G/K$.
Let us apply the integration formula \eqref{IntegrationFormula} to the function
\begin{align*}
f_{s}(xK)=\sum\nolimits_{\gamma\in\Gamma}e^{-s\dprime(\gamma{K},eK)}\mathbf{1}_{\Bprime_r(\gamma K)}(xK).
\end{align*}
On the one hand, as
\begin{align*}
\bigl|\dprime(xK,eK)-\dprime(\gamma K,eK)\bigr|\le r
\qquad\forall\,xK\in\Bprime_r(\gamma K),
\end{align*}
we have
\begin{align*}
\int_Xd(xK)\,f_{s}(xK)
\asymp\sum\nolimits_{\gamma\in\Gamma}e^{-s\dprime(\gamma K,eK)}
\underbrace{|\Bprime_r(\gamma K)|}_{|\Bprime_r(eK)|}
\asymp\Pprime_s(eK,eK).
\end{align*}
On the other hand,
\begin{align*}
\int_Xd(xK)\,f_{s}(xK)
&\lesssim\int_Xd(xK)\,e^{-s\dprime(xK,eK)}\\
&\asymp\int_{\mathfrak{a}^+}dH\,\omega(H)\,e^{-s\langle\frac\rho{\|\rho\|},H\rangle}
\lesssim\int_{\mathfrak{a}^+}dH\,e^{-(\frac s{\|\rho\|}-2)\langle\rho,H\rangle}
\end{align*}
is finite if $s>2\hspace{.2mm}\|\rho\|$.
Thus $\Pprime_s(eK,eK)<+\infty$ in that case, and consequently $\deltaprime(\Gamma)\le2\hspace{.2mm}\|\rho\|$.

\noindent(iii)
Denote the right hand side of \eqref{limsup} by $L(xK,yK)$
and let us first show that $L(xK,yK)$ is finite.
By applying Lemma \ref{ball disjoint} below to $y^{-1}\Gamma y$,
we deduce that there exists $r>0$ such that
the balls $\Bprime_r(\gamma yK)$, with $\gamma\in\Gamma$, are pairwise disjoint.
Set
\begin{align*}
\GammaRprime(xK,yK)=\lbrace\gamma\in\Gamma\,|\,\dprime(xK,\gamma yK)\le R\rbrace
\qquad\forall\,R\ge0,\,\forall\,x,y\in G.
\end{align*}
Then the ball $\Bprime_{R+r}(xK)$ contains the disjoint balls $\Bprime_r(\gamma yK)$,
with $\gamma\in\GammaRprime(xK,$ $yK)$. By computing volumes, we estimate
\begin{align*}
\NRprime(xK,yK)=|\GammaRprime(xK,yK)|\le\frac{|\Bprime_{R+r}(xK)|}{|\Bprime_r(eK)|}
\asymp(1+R)^{\ell-1}e^{2\hspace{.2mm}\|\rho\|R}.
\end{align*}
Hence $L(xK,yK)\le2\hspace{.2mm}\|\rho\|$.
Let us next show that $L(xK,yK)$ is actually independent of $x,y\in G$.
Given $x_1,y_1,x_2,y_2\in G$ and $R_1>0$, let
\begin{align*}
R_2=R_1+\dprime(x_1K,x_2K)+\dprime(y_1K,y_2K).
\end{align*}
Then the triangular inequality
\begin{align*}
\dprime(x_2K,\gamma y_2K)\le
\dprime(x_2K,x_1K)+\dprime(x_1K,\gamma y_1K)+\underbrace{\dprime(\gamma y_1K,\gamma y_2K)}_{\dprime(y_1K,\,y_2K)}
\end{align*}
implies successively
\begin{align*}
\Gamma_{\hspace{-.2mm}R_1}^{\hspace{.2mm}\prime}(x_1K,y_1K)
&\subset\Gamma_{\hspace{-.2mm}R_2}^{\hspace{.2mm}\prime}(x_2K,y_2K),\\
N_{\hspace{-.2mm}R_1}^{\hspace{.2mm}\prime}(x_1K,y_1K)
&\le N_{\hspace{-.2mm}R_2}^{\hspace{.2mm}\prime}(x_2K,y_2K),\\
L(x_1K,y_1K)
&\le L(x_2K,y_2K).
\end{align*}

\noindent Let us finally prove the equality between $\deltaprime(\Gamma)$ and $L=L(eK,eK)$.
For this purpose, observe that
\begin{align}
\Pprime_s
&=1+\sum\nolimits_{R\in\N^*}\sum\nolimits_{
\gamma\in\Gamma_{\!R}^{\hspace{.2mm}\prime}\smallsetminus\hspace{.2mm}\Gamma_{\!R-1}^{\hspace{.2mm}\prime}
}\!e^{-s\,\dprime(eK,\gamma K)}
\nonumber\\
&\asymp1 +\sum\nolimits_{R\in\N^*}\bigl(\NRprime-N_{\hspace{-.2mm}R-1}^{\hspace{.2mm}\prime}\bigr)\,e^{-sR}\nonumber\\
&\asymp\sum\nolimits_{R\in\N}\NRprime\,e^{-sR},
\label{series}
\end{align}
where we have written for simplicity
\begin{align*}
\Pprime_s=\Pprime_s(eK,eK),\quad
\GammaRprime=\GammaRprime(eK,eK)
\quad\text{and}\quad
\NRprime=\NRprime(eK,eK).
\end{align*}

\noindent One the one hand, let $s>L$ and set $\epsilon=\frac{s-L}2$.
By definition of $L$,
\begin{align*}
\NRprime\lesssim e^{(L+\epsilon)R}
\qquad\forall\,R\ge0.
\end{align*}
Hence
\begin{align*}
\Pprime_s\lesssim\sum\nolimits_{R\in\N}e^{-\epsilon R}<+\infty.
\end{align*}
One the other hand, let $s<L$.
By definition of $L$,
there exists a sequence of integers $1<R_1<R_2<\cdots\to+\infty$ such that
\begin{align*}
N_{\hspace{-.2mm}R_j}^{\hspace{.2mm}\prime}\ge e^{s R_j}
\qquad\forall\,j\in\N^*.
\end{align*}
Hence the series \eqref{series} diverges.
\end{proof}

\begin{remark}
Here is an example where \,$\delta\ssb<\ssb\delta^{\ssf\prime}$.
Consider the product
\vspace{.5mm}

\centerline{$
\Gamma\msf\backslash\ssf G/K\ssb
=(\ssf\Gamma_{\ssb1}\backslash\ssf G_{\vsf\vsf1}/K_{\ssb1})\msb
\times\msb(\ssf\Gamma_{\ssb2}\backslash\ssf G_{\ssf2\vsf}/K_{\vsb2})
$}\vspace{.5mm}
\noindent 
of two locally symmetric spaces of rank one,
with parameters \msf$\rho_1$, $\delta_1$ and \msf$\rho_2$\vsf, $\delta_2$\ssf.
Then

\begin{equation}\label{EstimateDelta}
\delta\le\sqrt{\vsf\delta_1^{\ssf2}\!+\ssb\delta_2^{\ssf2}\msf}
\end{equation}
and
\begin{equation}\label{EstimateDeltaPrime}
\delta^{\ssf\prime}
\ge\sqrt{\rho_1^{\ssf2}\!+\ssb\rho_2^{\ssf2}\ssf}\,
\max\,\bigl\{\tfrac{\delta_1}{\rho_1},\tfrac{\delta_2}{\rho_2}\bigr\}\msf.
\end{equation}
Hence \,$\delta\ssb<\ssb\delta^{\ssf\prime}$ if \,$\tfrac{\delta_1}{\rho_1}\msb\ne\msb\tfrac{\delta_2}{\rho_2}$\vsf.
Notice that there are plenty of such products,
starting with the case where \,$\delta_1\!=\ssb2\ssf\rho_1$ and \,$\delta_2\msb=\ssb0$\ssf.
Let us first prove \eqref{EstimateDelta} and begin with some notation.
Write for simplicity
\vspace{.5mm}

\centerline{
$N_{1,\vsf R}\ssb=\msb(N_1)_{\vsb R\ssf}(e\vsf K_{\vsb1},e\vsf K_{\vsb1})$\ssf,
\ssf$N_{2\vsf,\vsf R}\ssb=\msb(N_2)_{\vsb R\ssf}(e\vsf K_{\vsb2},e\vsf K_{\vsb2})$
\ssf and \,$N_R\ssb=\msb N_{\vsb R\ssf}(e\vsf K,e\vsf K)$\ssf,
}\vspace{.5mm}

\noindent
for every \msf$R\ssb\ge\ssb0$\ssf.
Moreover, for every \msf$D\msb\subset\ssb\smash{\mathbb{R}_+^2}$,
denote by \ssf$N(D)$ the number of \,$\gamma\ssb=\ssb(\gamma_1,\gamma_{\vsf2})$
in \,$\Gamma\msb=\ssb\Gamma_{\ssb1}\msb\times\ssb\Gamma_{\ssb2}$ \ssf such that
\msf$\bigl({d}_{\vsf1\vsb}(\gamma_1K_{\ssb1},e\vsf K_{\ssb1})\vsf,{d}_2(\gamma_{\vsf2}K_{\vsb2},e\vsf K_{\vsb2})\ssb\bigr)$
belongs to $D$.
In \,$\mathbb{R}_+^2$ we consider the covering of
\msf$D_R\ssb=\ssb\{\vsf(R_1\ssb,R_2)\msb\in\msb\mathbb{R}_+^2\msf|\msf R_1^2\msb+\msb R_2^2\ssb\le\msb R^2\vsf\}$
by the two segments \,$[\ssf0\ssf,\ssb R\ssf]\msb\times\msb\{0\}$\vsf, $\{0\}\msb\times\msb[\ssf0\ssf,\ssb R\ssf]$
and by the squares \,$Q_{j_1\vsb,\ssf j_2}\msb=\ssb(j_1\vsb,j_1\!+\!1]\ssb\times\msb(j_2,j_2\!+\!1]\ssf$,
with \,$j_1^2\msb+\ssb j_2^2\msb<\msb R^2$.
Then

\begin{equation}\label{Inequality1}
N_R=N(D_R)\le N_{1,\vsf R}+N_{2\vsf,\vsf R}+\ssb\sum\nolimits_{\vsf j_1^2\vsb+j_2^2<R^2}\ssb N(Q_{j_1\vsb,\ssf j_2})\msf,
\end{equation}

\noindent
with
\begin{equation}\label{Inequality2}
N(Q_{j_1\vsb,\ssf j_2})\ssb
=\ssb(N_{1,\ssf j_1\vsb\vsb+1}\!-\msb N_{1,\ssf j_1}\vsb)\ssf(N_{2\vsf,\ssf j_2\vsb+1}\!-\msb N_{2\vsf,\ssf j_2})\ssb
\le\vsb N_{1,\ssf j_1\vsb\vsb+1}\ssf N_{2\vsf,\ssf j_2\vsb+1}\vsf.
\end{equation}

\noindent
Given \,$s_1\!>\ssb\delta_1$ and \,$s_2\msb>\ssb\delta_2$\ssf,
there exist \,$C\msb\ge\msb1$ such that
\begin{equation}\label{Inequality3}
N_{1,\vsf R}\ssb\le\ssb C\ssf e^{\ssf s_1R}
\quad\text{and}\quad
N_{2\vsf,\vsf R}\ssb\le\ssb C\ssf e^{\ssf s_2\vsf R}
\end{equation}

\noindent
for every \,$R\ssb\ge\ssb0$\ssf.
By combining \eqref{Inequality1}, \eqref{Inequality2} and \eqref{Inequality3}, we get
\begin{equation}\label{Inequality4}
N_R\le C\ssf e^{\ssf s_1R}+C\ssf e^{\ssf s_2\vsf R}
+C^{\vsf2}\sum\nolimits_{\vsf j_1^2\vsb+j_2^2<R^2}\ssb e^{\ssf s_1\vsb(j_1\ssb+1)\vsf+\vsf s_2(j_2\vsb+1)}\msf.
\end{equation}

\noindent
Up to a multiplicative constant, the right hand side of \eqref{Inequality4} is bounded above by the integral
\begin{equation}\label{Inequality5}
\int_{\ssf\mathbb{R}_+^2\vsb\vsb\cap B(0\vsf,R+2)}\hspace{-1mm}d\vsf R_1\ssf d\vsf R_2\,e^{\ssf s_1\vsb\vsb R_1\vsb+\ssf s_2R_2}
=\ssb\int_{\ssf0}^{R\vsf+2}\hspace{-1mm}dr\,r\int_{\ssf0}^{\frac\pi2}\!d\theta\,e^{\ssf r\ssf(s_1\ssb\cos\theta\ssf+\ssf s_2\sin\theta)}\msf.
\end{equation}

\noindent
As the function \,$\theta\vsb\longmapsto\ssb s_1\ssb\cos\theta\ssb+\ssb s_2\sin\theta$ \ssf
reaches its maximum $\sqrt{\vsb s_1^2\msb+\ssb s_2^2\ssf}$ at \,$\theta_0\ssb=\ssb\arctan\frac{s_2}{s_1}$,
the latter integral is itself bounded above by
\begin{equation}\label{Inequality6}
\tfrac\pi2\int_{\ssf0}^{R\vsf+2}\hspace{-1mm}dr\,r\,e^{\ssf r\vsf\sqrt{\ssb s_1^2\vsb+s_2^2}}
\le\tfrac\pi2\ssf\tfrac{R+2}{\sqrt{s_1^2\ssb+s_2^2\ssf}}\,e^{\ssf(\vsb R+2)\sqrt{\ssb s_1^2\vsb+s_2^2}}
\end{equation}

\noindent
In conclusion, we obtain
\vspace{1mm}

\centerline{$
\tfrac{\log N_R}R\le
\tfrac{2\log C\vsf+\ssf\log\pi\vsf-\ssf\log2\ssf-\frac12\ssb\log\vsf(\vsb s_1^2+s_2^2)}R+\tfrac{\log\vsf(\vsb R\vsf+\vsf2)}R+\tfrac{R\vsf+2}R\ssf\sqrt{\ssb s_1^2\!+\msb s_2^2\ssf}
$}\vspace{1.5mm}

\noindent
by combining \eqref{Inequality4}, \eqref{Inequality5} and \eqref{Inequality6},
hence \,$\delta\ssb\le\msb\sqrt{\ssb s_1^2\msb+\ssb s_2^2\ssf}$ \ssf by letting \,$R\msb\to\msb+\infty$\ssf,
and finally \,$\delta\ssb\le\msb\sqrt{\ssb\delta_1^2\!+\ssb\delta_2^2\ssf}$ \ssf
by letting \,$s_1\!\searrow\ssb\delta_1$ and \,$s_2\!\searrow\ssb\delta_2$\vsf.
Let us turn to the proof of \eqref{EstimateDeltaPrime}.
Set \,$\|\rho\|\ssb=\ssb\sqrt{\rho_1^{\ssf2}\!+\ssb\rho_2^{\ssf2}\ssf}$
and assume that \,$\tfrac{\delta_1}{\rho_1}\msb\ge\msb\tfrac{\delta_2}{\rho_2}$\vsf. As

\centerline{$
d^{\ssf\prime}(\gamma\ssf K,e\vsf K)\ssb
=\ssb\tfrac{\rho_1}{\|\rho\|}\,{d}_1(\gamma_1\vsf K_{\ssb1},e\vsf K_{\ssb1})\ssb
+\ssb\tfrac{\rho_2}{\|\rho\|}\,{d}_2(\gamma_2\vsf K_{\vsb2},e\vsf K_{\vsb2})\msf,
$}\vspace{.5mm}

\noindent
the set \,$\{\ssf\gamma\msb\in\msb\Gamma\,|\,d^{\ssf\prime}(\gamma\ssf K,e\vsf K)\ssb\le\ssb R\ssf\}$ contains the product
\vspace{.5mm}

\centerline{$
\{\ssf\gamma_{\vsf1}\hspace{-.8mm}\in\msb\Gamma_{\ssb1}\msf|\,{d}_1(\gamma_1\vsb K_{\ssb1},e\vsf K_{\ssb1})\msb
\le\msb\tfrac{\|\rho\|}{\rho_1}R\ssf\}\ssb\times\ssb\{e\}\msf,
$}

\noindent
for every \,$R\ssb\ge\ssb0$\ssf.
Hence \msf$N_{\ssb R}^{\ssf\prime}\ssb\ge\ssb N_{1\vsf,\smash{\frac{\|\rho\|}{\rho_1}}\vsb R}^{\vphantom{|}}$\msf,
where \,$N_{\ssb R}^{\ssf\prime}\ssb=\msb N_{\vsb R\ssf}^{\ssf\prime}(e\vsf K,e\vsf K)$\vsf,
and conseqently \,$\delta^{\ssf\prime}\msb\ge\ssb\tfrac{\|\rho\|}{\rho_1}\ssf\delta_1$.
\end{remark}

\begin{lemma}\label{ball disjoint}
\footnote{As observed by the referee, Lemma 3 still holds without the torsion-free assumption, provided that $\gamma$ runs through $\Gamma\backslash(\Gamma\cap{K})$.}
There exists $r>0$ such that the balls
$\Bprime_{r}(\gamma K)$, with $\gamma\in\Gamma$,
are pairwise disjoint in $G/K$.
\end{lemma}

\begin{proof}
Let $r>0$. As $\Gamma$ is discrete in $G$, its intersection with the compact subset
\begin{align*}
\Gprime_r=\lbrace y\in G\,|\,\dprime(yK,eK) \le r\rbrace=K{\bigl(}\exp
\lbrace H\in\overline{\mathfrak{a}^+}\,|\,{\langle\rho,H\rangle\le\|\rho\|r}\rbrace \bigr)K
\end{align*}
is finite. Moreover, as $\Gamma$ is torsion-free,
\begin{align*}
\gamma^+\neq0\qquad\forall\,\gamma\in\Gamma\backslash\lbrace e\rbrace.
\end{align*}
Hence there exists $r>0$ such that $\Gamma\cap\Gprime_{2r}=\lbrace e\rbrace$,
which implies that the sets $\gamma\Gprime_r$ are pairwise disjoint in $G$.
In other words, the balls $\Bprime_r(\gamma K)$ are pairwise disjoint in $G/K$.
\end{proof}

\vspace{10pt}
By using $\deltaprime{(\Gamma)}$,
we prove Theorem \ref{TheoremFirstImprovement},
which improves the lower bound in Theorem \ref{leuzinger}.

\begin{proof}[Proof of Theorem \ref{TheoremFirstImprovement}]
Let us resume the approach in \cite[Section 4]{Cor1990} and \cite[Section 3]{Leu2004}.
It consists in studying the convergence of the positive series
\begin{align}\label{GreenSeries}
g_\zeta^{\Gamma} (\Gamma xK,\Gamma yK)=\sum\nolimits_{\gamma\in\Gamma}g_\zeta(Ky^{-1}\gamma^{-1}xK),
\end{align}
which expresses the kernel $g_\zeta^\Gamma$ of $(-\Delta-\|\rho\|^2+\zeta^2)^{-1}$
on the locally symmetric space $Y\!=\Gamma\backslash G/K$ in terms of
the corresponding Green function $g_\zeta$ on the symmetric space $X=G/K$.
Here $\zeta>0$ and $\Gamma xK\neq\Gamma yK$.
Recall \cite[Theorem 4.2.2]{AnJi1999} that
\begin{align}\label{GreenFunction}
g_\zeta(\exp H)
\asymp\Bigl\{\prod\nolimits_{\alpha\in\Sigma_{\text{red}}^+}\!\bigl(1+\langle\alpha,H\rangle\bigr)\Bigr\}\,
\|H\|^{-\frac{\ell-1}2-|\Sigma_{\text{red}}^+|}\,e^{-\langle\rho,H\rangle-\zeta\|H\|}
\end{align}
for $H\in\overline{\mathfrak{a}^+}$ large, let say $\|H\|\ge\frac12$,
while
\begin{align*}
g_\zeta(\exp H)\asymp\begin{cases}
\|H\|^{-(n-2)}
&\text{if }\,n>2\\
\log\frac1{\|H\|}
&\text{if }\,n=2\\
\end{cases}\end{align*}
for $H$ small, lets say $0<\|H\|\le\frac12$.
Thus \eqref{GreenSeries} converges if and only if
\begin{equation}\label{SeriesGamma}\begin{aligned}
\sum\nolimits_{\gamma\in\Gamma}\,
&\Bigl\{\prod\nolimits_{\alpha\in\Sigma_{\text{red}}^+}\!
\bigl(1\!+\hspace{-.5mm}\langle\alpha,(y^{-1}\gamma^{-1}x)^+\rangle\bigr)\Bigr\}\,\times\\
&\times\,d(xK,\gamma yK)^{-\frac{\ell-1}2-|\Sigma_{\text{red}}^+|}\,
e^{-\|\rho\|\dprime(xK,\gamma yK)-\zeta d(xK,\gamma yK)}
\end{aligned}\end{equation}
converges.
Let us compare the series \eqref{SeriesGamma}
with the Poincar\'e series \eqref{PSeries} and \eqref{PprimeSeries}.
On the one hand, as $\|(y^{-1}\gamma^{-1}x)^+\|=d(xK,\gamma yK)$,
\eqref{SeriesGamma} is bounded from above by
$\Pprime_{\|\rho\|+\zeta}(xK, yK)$.
On the other hand, as
\begin{align*}
d(xK,\gamma yK)^{-\frac{\ell-1}2-|\Sigma_{\text{red}}^+|}
\gtrsim e^{-\epsilon d(xK,\gamma yK)}
\end{align*}
for every $\epsilon>0$, \eqref{SeriesGamma} is bounded from below by
$P_{\|\rho\|+\zeta+\epsilon}(xK, yK)$.
Hence \eqref{SeriesGamma} converges
if $\|\rho\|+\zeta>\deltaprime{(\Gamma)}$,
i.e., $\zeta>\deltaprime{(\Gamma)}-\|\rho\|$,
while $\eqref{SeriesGamma}$ diverges
if $\zeta<\delta{(\Gamma)}-\|\rho\|$.
We conclude by using the fact \cite[Section 4]{Cor1990} that
$\lambda_0(Y)$ is the supremum of $\|\rho\|^2-\zeta^2$
over all $\zeta>0$ such that \eqref{GreenSeries} converges.
\end{proof}

Next statement is obtained by combining this lower bound with the upper bound in Theorem \ref{leuzinger}.

\begin{corollary}\label{CorollaryFirstImprovement}
The following estimates hold for $\lambda_0(Y)$:
\begin{align*}\begin{cases}
\,\lambda_0(Y)=\|\rho\|^2
&\qquad\textnormal{if}\;\;\deltaprime{(\Gamma)}\le\|\rho\|,\\
\,\|\rho\|^2-(\deltaprime{(\Gamma)}-\|\rho\| )^2\le\lambda_0(Y)\le\|\rho\|^2
&\qquad\textnormal{if}\;\;
\delta{(\Gamma)}\le\|\rho\|\le\deltaprime{(\Gamma)},\\
\,\|\rho\|^2-(\deltaprime{(\Gamma)}-\|\rho\|)^2
\le\lambda_0(Y)\le\|\rho\|^2-(\delta{(\Gamma)}-\|\rho\| )^2
&\qquad\textnormal{if}\;\;\|\rho\|\le\delta{(\Gamma)}.
\end{cases}
\end{align*}
\end{corollary}
\section{Second improvement}

In this section, we obtain the actual higher rank analog of Theorem \ref{corlette} by considering a further family of distances on $X$,
which reflects the large scale behavior \eqref{GreenFunction} of the Green function.
Specifically, for every $s>0$ and $x,y\in G$, let
\begin{equation}\label{ds}\begin{aligned}
\dsecond_{s}(xK,yK)
&=\min\lbrace s,\|\rho\|\rbrace\,\dprime(xK,yK)
+\max\lbrace s-\|\rho\|,0\rbrace\,d(xK,yK)\\
&=\begin{cases}
\,s\,\dprime(xK,yK)
&\text{if }\,0<s\le\|\rho\|,\\
\,\|\rho\|\,\dprime(xK,yK)+(s-\|\rho\|)\,d(xK,yK)
&\text{if }\,s\ge\|\rho\|.\\
\end{cases}\end{aligned}\end{equation}
Then \eqref{ds} defines a $G$-invariant distance on $X$ such that
\begin{align}\label{d' ds d}
s\,\dprime(xK,yK)
\le \dsecond_{s}(xK,yK)
\le s\,d(xK,yK)
\qquad\forall\,s>0,\,\forall\,x,y\in G.
\end{align}
Consider the associated Poincar\'e series
\begin{align}\label{PsecondSeries}
\Psecond_s(xK, yK) =\sum\nolimits_{\gamma\in\Gamma}e^{-\dsecond_{s}(xK,\gamma yK)}
\qquad\forall\,s>0,\,\forall\,x,y\in G
\end{align}
and its critical exponent
\begin{align*}
\deltasecond{(\Gamma)}=\inf\lbrace s>0\,|\,\Psecond_s(xK,yK)<+\infty\rbrace.
\end{align*}
It follows from \eqref{d' ds d} that
\begin{align}\label{delta delta'' delta'}
0\le\delta{(\Gamma)}
\le\deltasecond{(\Gamma)}
\le\deltaprime{(\Gamma)}\le2\hspace{.2mm}\|\rho\|.
\end{align}

\begin{proof}[Proof of Theorem \ref{second improvement}]
In the proof of Theorem \ref{TheoremFirstImprovement}
and Corollary \ref{CorollaryFirstImprovement},
we compared the series \eqref{GreenSeries}, or equivalently \eqref{SeriesGamma},
with the Poincar\'e series \eqref{PSeries} and \eqref{PprimeSeries}.
If we consider instead the Poincar\'e series \eqref{PsecondSeries},
we obtain in the same way that
\eqref{SeriesGamma} is bounded from above by $\Psecond_{\|\rho\|+\zeta}(xK, yK)$
and from below by $\Psecond_{\|\rho\|+\zeta+\epsilon}(xK, yK)$,
for every $\epsilon>0$.
Hence \eqref{SeriesGamma} converges if $\zeta>\deltasecond(\Gamma)-\|\rho\|$,
while $\eqref{SeriesGamma}$ diverges if $\zeta<\deltasecond(\Gamma)-\|\rho\|$.
We conclude as in the above-mentioned proof.
\end{proof}


\vspace{10pt}
\noindent\textbf{Acknowledgements.}
The authors are grateful to the referee for checking carefully the manuscript and making several helpful suggestions of improvement. The second author acknowledges financial support from the University of Orléans during his Ph.D. and from the Methusalem Programme \textit{Analysis and Partial Differential Equations} during his postdoc stay at Ghent University.

\printbibliography

\vspace{20pt}
\address{
\noindent\textsc{Jean-Philippe Anker:}
\href{mailto:anker@univ-orleans.fr}
{anker@univ-orleans.fr}\\
Institut Denis Poisson,
Universit\'e d'Orl\'eans, Universit\'e de Tours \& CNRS,
Orl\'eans, France}
\vspace{10pt}

\address{
\noindent\textsc{Hong-Wei Zhang:}
\href{mailto:hongwei.zhang@ugent.be}
{hongwei.zhang@ugent.be}\\
Department of Mathematics,
Ghent University,
Ghent, Belgium\\
Institut Denis Poisson,
Universit\'e d'Orl\'eans, Universit\'e de Tours \& CNRS,
Orl\'eans, France}
\end{document}